\documentclass[12pt]{article}
\usepackage{amsthm}
\usepackage{amsmath, amssymb,enumerate}
\usepackage{graphicx}
\usepackage{amsfonts}
\usepackage{hyperref}
\usepackage{tikz}

\newtheorem{theorem}{Theorem}[section]
\newtheorem{corollary}[theorem]{Corollary}

\newtheorem{lemma}[theorem]{Lemma}
\newtheorem{proposition}[theorem]{Proposition}

\theoremstyle{definition}

\newtheorem{example}[theorem]{Example}

\def\qqq{\mathbb{Q}}

\def\zzz{\mathbb{Z}}

\def\K{\mathcal{K}}

\def\R{R}
\def\S{S}

\DeclareMathOperator{\Diag}{Diag}
\DeclareMathOperator{\divisor}{div}
\DeclareMathOperator{\Div}{Div}
\DeclareMathOperator{\Ram}{Ram}
\DeclareMathOperator{\Prin}{Prin}
\DeclareMathOperator{\Star}{Star}
\DeclareMathOperator{\Jac}{Jac}
\usepackage{authblk}
\begin{document}

\author[1]{Kassie Archer}
\author[1]{Caroline Melles}
\affil[1]{\small{Department of Mathematics, U.~S.~Naval Academy, Annapolis, MD, 21402, USA 

Email: karcher@usna.edu, cgg@usna.edu}}
\title{Harmonic Morphisms of Arithmetical Structures on Graphs}
\date{}
\maketitle

\begin{abstract}
Let $\phi \colon \Gamma_2 \rightarrow \Gamma_1$ be a 
harmonic morphism of connected graphs. We show that an arithmetical 
structure on $\Gamma_1$ can be pulled back via $\phi$ to an 
arithmetical structure on $\Gamma_2$. 
%The generalized Laplacian associated to each arithmetical 
%structure determines an abelian group, the arithmetical critical group. 
We then show that some results of Baker and Norine 
on the critical groups for the usual Laplacian 
extend to arithmetical critical groups, which are abelian groups determined by the generalized Laplacian associated to these arithmetical structures.
In particular, we show that the morphism $\phi$ induces a surjective group 
homomorphism from the arithmetical critical group of $\Gamma_2$
to that of $\Gamma_1$ and an injective group homomorphism 
from the arithmetical critical group of $\Gamma_1$ to 
that of $\Gamma_2$. Finally, we prove a Riemann-Hurwitz formula 
for arithmetical structures.
\end{abstract}



\section{Introduction}

Given two finite graphs $\Gamma_1$ and $\Gamma_2$, a harmonic morphism $\phi
\colon \Gamma_2\to\Gamma_1$ is a morphism that preserves locally harmonic 
functions on the vertices, as originally defined by Urakawa \cite{Ur00}. 
Baker and Norine \cite{BN09} argued that these morphisms are an appropriate discrete analogue of holomorphic maps between Riemann surfaces and justified this with results analogous to those from algebraic geometry. They 
showed that a harmonic morphism $\phi$ induces certain homomorphisms 
 between the Jacobians (or critical groups) of the graphs. 
 In particular, they showed that the pushforward  
 $\phi_* \colon \Jac(\Gamma_2)\to\Jac(\Gamma_1)$
 is a surjective homomorphism, and the pullback 
 $\phi^* \colon \Jac(\Gamma_1)\to\Jac(\Gamma_2)$ is an injective one. 
 
 Baker and Norine also proved a Riemann-Hurwitz formula for graphs. This establishes a relationship 
 between the genus (or first Betti number) 
 $g_1$ of $\Gamma_1$ and the genus $g_2$ of $\Gamma_2$, namely that 
 \[ 2g_2 -2 = \deg(\phi)(2g_1-2) + \sum_{v\in V(\Gamma_2)} (2\mu(v)-2+\nu(v))\]
 where $\deg(\phi)$ is the degree of the harmonic morphism, 
 and $\mu(v)$ and $\nu(v)$ are the horizontal and vertical multiplicities at a vertex $v\in V(\Gamma_2)$ (see Section~\ref{sec:harmonic}).

The aim of this paper is to extend these results of Baker and Norine about harmonic morphisms between graphs to the case of harmonic morphisms between so-called arithmetical structures on graphs. 
Arithmetical structures on finite graphs were originally introduced by Lorenzini in 1989 \cite{Lo89}, and are described in this paper in Section~\ref{sec:arith structures}.
An arithmetical structure on a graph $\Gamma$ determines an arithmetical Laplacian, which is a generalization of
the standard graph Laplacian.
This arithmetical Laplacian can be used to compute the arithmetical critical group. Lorenzini's motivation behind defining these structures came from these arithmetical Laplacian matrices, which originally appeared as the intersection matrices for certain degenerating curves in algebraic geometry. The corresponding arithmetical critical group is known as the group of components of the N\'{e}ron model in algebraic geometry (see \cite{Lo90}).

In Section~\ref{sec:main}, we show that one can pull back an arithmetical structure on $\Gamma_1$ to an arithmetical structure on $\Gamma_2$ via a harmonic morphism 
$\phi \colon \Gamma_2\to\Gamma_1$. This harmonic morphism then induces homomorphisms between their respective arithmetical critical groups, giving results similar to those of Baker and Norine. In Section~\ref{sec:RH}, we describe the arithmetical genus of a graph and prove a Riemann-Hurwitz formula for arithmetical structures on graphs.

\section{Background and Notation}

In this paper, we consider connected simple graphs on $n\geq 2$ vertices with
 no loops, and denote such a graph with $\Gamma=(V,E)$, 
 where $V=V(\Gamma)$ is the 
 vertex set of $\Gamma$ and $E=E(\Gamma)$ is the non-empty edge set of $\Gamma$. 
If $v$ and $w$ are adjacent vertices in $\Gamma$ (i.e., if $(v,w)\in E$), we write $v \sim w$.

Let $\Gamma_2 = (V_2, E_2)$ and $\Gamma_1=(V_1,E_1)$ be 
connected simple graphs. A {\it graph morphism} $\phi \colon 
\Gamma_2 \rightarrow \Gamma_1$ is a pair of maps 
$V_2 \rightarrow V_1$ and $E_2 \rightarrow V_1 \cup E_1$ 
such that if $e=(v,w)$ is an edge in $\Gamma_2$, then either:
\begin{enumerate}[(i)]
\item 
$\phi(v) = \phi(w)$ and $\phi(e) = \phi(v) = \phi(w)$, or 
\item  
$\phi(v)$ and $\phi(w)$ are adjacent in $\Gamma_1$, and 
$\phi(e)$ is 
%equal to 
the edge $(\phi(v), \phi(w))$ in $\Gamma_1$.
\end{enumerate}
In case (i), the edge $e\in E_2$ is mapped to a vertex in $V_1$, and we say that the edge $e$ is {\it vertical for $\phi$}, and 
in case (ii), the edge $e \in E_2$ maps to another edge in $E_1$, and we say that the edge $e$ is {\it horizontal for $\phi$}.

\subsection{Harmonic morphisms}\label{sec:harmonic}

Let $\Gamma_1 = (V_1, E_1)$ and $\Gamma_2 = (V_2, E_2)$ 
be connected graphs with non-empty edge sets and with $n_1=|V_1|$ and $n_2=|V_2|$. 
Let $\phi \colon \Gamma_2 \rightarrow \Gamma_1$ 
be a graph morphism.
%If an edge in $\Gamma_2$ maps to a vertex in $\Gamma_1$
%under $\phi$, it is called a {\it vertical edge} for $\phi$.
If $v$ is a vertex in $\Gamma_2$, the {\it vertical multiplicity} 
of $\phi$ at $v$ is defined to be the number of vertical edges 
incident to $v$, and is denoted $\nu_\phi(v)$ or simply $\nu(v)$.
If $v$ is a vertex in $\Gamma_2$ and $f$ is an edge in $\Gamma_1$ 
that is incident to $\phi(v)$, 
the {\it local horizontal multiplicity} $\mu_\phi(v, f)$ 
is defined to be the  number of edges of $\Gamma_2$ 
that are incident to $v$ and map to $f$.

We say that $\phi$ is {\it harmonic at $v$} if 
$\mu_\phi(v,f)$ takes the same value for all edges $f$ 
in $\Gamma_1$ that are incident to $\phi(v)$.
We say that $\phi$ is a {\it harmonic morphism} if 
$\phi$ is harmonic at every vertex in $\Gamma_2$.
In this case, the {\it horizontal multiplicity of $\phi$ at $v$}
is defined to be the value of $\mu_\phi(v,f)$ for any 
edge $f$ incident to $\phi(v)$, and is denoted 
$\mu_\phi(v)$, or simply $\mu(v)$. 

A constant (or trivial) morphism, mapping $\Gamma_2$ 
to a single vertex of $\Gamma_1$, is always harmonic.
Unless otherwise stated, we will assume that all graph morphisms are 
non-constant (i.e., non-trivial).

\begin{example}
\label{ex:C3-W5-part1}
Let $\Gamma_1 = C_3$ be the 3-cycle graph with vertex set $V_1 = \{x_0, x_1, x_2\}$ 
and edge set $E_1 = \{(x_0, x_1), (x_1, x_2), (x_2, x_0)\}$. 
Let $\Gamma_2$ be the wheel graph $W_5$ on 5 vertices, with 
$V_2 = \{ v_0,  v_1,v_2,v_3 , v_4\}$ where $v_0$ is the central vertex, 
so that the edge set is \[E_2=\{ (v_0, v_1), (v_0, v_2),(v_0, v_3), (v_0, v_4), 
(v_1, v_2), (v_2, v_3), (v_3, v_4), (v_4, v_1)\}.\]
Let $\phi \colon \Gamma_2 \rightarrow \Gamma_1$ 
be the graph morphism determined by the following rule.
\begin{center}
\begin{tabular}{r|r|r|r|r|r}
$v_i$ & $v_0$ & $v_1$& $v_2$& $v_3$& $v_4$\\
\hline 
$\phi(v_i)$ & $x_0$& $x_1$& $x_1$& $x_2$& $x_2$
\end{tabular}
\end{center}

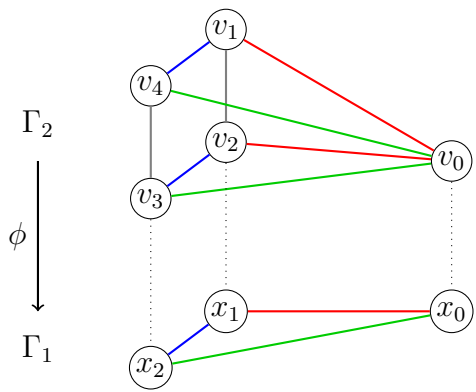
\begin{figure}[h]
\begin{minipage}{\textwidth}
\begin{center}
%\vspace{8.0 cm}
\begin{tikzpicture}
\node[circle, draw,inner sep=1pt] (x1) at (0,0) {$x_0$};
\node[circle, draw,inner sep=1pt] (x2) at (-3,0) {$x_1$};
\node[circle, draw,inner sep=1pt] (x3) at (-4,-.75) {$x_2$};

\node[circle, draw,inner sep=1pt] (v1) at (0,2) {$v_0$};
\node[circle, draw,inner sep=1pt] (v2) at (-3,3.75) {$v_1$};
\node[circle, draw,inner sep=1pt] (v3) at (-3,2.25) {$v_2$};
\node[circle, draw,inner sep=1pt] (v5) at (-4,3) {$v_4$};
\node[circle, draw,inner sep=1pt] (v4) at (-4,1.5) {$v_3$};

 \draw [thick, red] (x1) -- (x2);
 \draw[thick, blue] (x2)--(x3);
 \draw [thick, green!80!black] (x3)--(x1);
  \draw  [thick, red](v1)--(v2);
  \draw  [thick, gray] (v2)-- (v3);
  \draw  [thick, blue] (v3)--(v4);
  \draw [thick, blue] (v5)--(v2);
  \draw  [thick, red] (v3)--(v1);
  \draw [thick, green!80!black] (v1)--(v4);
  \draw  [thick, green!80!black] (v1)--(v5);
  \draw [thick, gray] (v4)--(v5);
  
  \draw[dotted] (v4)--(x3);
  \draw[dotted] (v3)--(x2);
  \draw[dotted] (v1)--(x1);
  
  \node at (-5.5, 2.5) {$\Gamma_2$};
    \node at (-5.5, -.5) {$\Gamma_1$};
    \draw [thick,->] (-5.5,2)-- node[left] {$\phi$} ++ (0,-2) ;
  
\end{tikzpicture}
\end{center}
\end{minipage}
\caption{A harmonic morphism $\phi$ from $\Gamma_2=W_5$ to $\Gamma_2=C_3$.}
\label{fig:harmonic}
\end{figure}

The graph morphism $\phi$, illustrated in Figure~\ref{fig:harmonic}, is harmonic. 
Notice that the edges $(v_1,v_2)$ and $(v_3,v_4)$ in $\Gamma_2$ are vertical since 
$\phi(v_1) = \phi(v_2) = x_1$ and $\phi(v_3) = \phi(v_4) = x_2$, while all other edges are horizontal.
The vertex $v_0$ has horizontal multiplicity 2 and vertical 
multiplicity 0, and all other 
vertices of $\Gamma_2$ have horizontal multiplicity 1 and vertical 
multiplicity 1.
\end{example}

%\todo{Maybe we should put the first part of example 3.2 here, and refer back to %it when we talk about it in Section 3?}

Harmonic morphisms also have a nice description in terms of 
matrices. First fix orderings 
%$x_1, \dots , x_{n_1}$ 
of the vertices of $\Gamma_1$ 
and 
%$v_1, \dots , v_{n_2}$ 
of the vertices of $\Gamma_2$, and let $A_1$ and $A_2$ be the adjacency matrices
of $\Gamma_1$ and $\Gamma_2$ with respect 
to these orderings. 
The graph morphism $\phi$ can be described by an $n_2\times n_1$ zero-one 
vertex map matrix $\Phi$, with rows indexed by the vertices of $\Gamma_2$ 
and columns indexed by the vertices of $\Gamma_1$, such that 
\[
\Phi_{v,x} = 
\begin{cases} 
1 &\quad \text{if $\phi(v)=x$,}\\
0 &\quad \text{otherwise.}
\end{cases}
\]

We define a vector $\boldsymbol \mu$ of horizontal multiplicities 
and a vector $\boldsymbol \nu$ of vertical multiplicities of the 
vertices of $\Gamma$.
The horizontal multiplicity matrix $D_\mu=\Diag(\boldsymbol \mu)$
 is the diagonal matrix whose diagonal elements are the horizontal 
 multiplicities of vertices of $\Gamma_2$. 
 Similarly, the vertical multiplicity matrix $D_\nu=\Diag(\boldsymbol \nu)$
 is the diagonal matrix whose diagonal elements are the vertical multiplicities.

\begin{proposition}
(Melles-Joyner \cite[Theorem~1]{MJ24})
%\cite{JM17})
\label{propn:harmonic-adjacency-criterion}
If $\phi \colon \Gamma_2 \rightarrow \Gamma_1$ is a 
non-constant harmonic morphism of connected graphs, then
\[
A_2 \Phi = D_\nu \Phi + D_\mu \Phi A_1.
\]
\end{proposition}

In fact, the identity in Proposition \ref{propn:harmonic-adjacency-criterion}
characterizes harmonic morphisms (see \cite{MJ24}), but we will not need this fact here.
See also \cite[Section~3.3.2]{JM17}
for an identity relating the incidence
matrices and an identity relating the 
Laplacian matrices of a harmonic morphism.

Urakawa \cite{Ur00} proved that under a non-constant harmonic 
morphism, the number of preimages of an 
edge of $\Gamma_1$ is the same for all edges in $\Gamma_1$. 
This number is called the {\it degree} of $\phi$ and denoted 
$\deg(\phi)$.

\begin{example}\label{ex:C3-W5-part1.5}
Let $\phi \colon \Gamma_2 \rightarrow \Gamma_1$ 
be the harmonic morphism 
described in Example \ref{ex:C3-W5-part1}, 
mapping the wheel graph on 5 vertices to the 3-cycle graph.
The adjacency matrices of $\Gamma_1$ and $\Gamma_2$ 
and the vertex map matrix are
\[
A_1 = 
\left(
\begin{array}{rrr}
0&1&1\\ 1&0&1 \\ 1&1&0
\end{array} \right), \,
A_2 = 
\left(
\begin{array}{rrrrr}
0&1&1&1&1\\ 1&0&1&0&1\\
1&1&0&1&0 \\ 1&0&1&0&1 \\ 1& 1&0&1&0
\end{array} \right),
\, \text{and} \,
\Phi = 
\left( \begin{array}{rrr}
1&0&0 \\ 0&1&0 \\ 0&1&0 \\ 0&0&1 \\ 0&0&1
\end{array} \right).
\]
%The vertex map matrix is given by
%\[
%\Phi = 
%\left( \begin{array}{rrr}
%1&0&0 \\ 0&1&0 \\ 0&1&0 \\ 0&0&1 \\ 0&0&1
%\end{array} \right).
%\]
The horizontal and vertical multiplicity matrices are 
$D_\mu = \Diag(\boldsymbol \mu)$ and $D_\nu = \Diag(\nu)$, 
where $\boldsymbol \mu = (2,1,1,1,1)^t$, and 
$\boldsymbol \nu =(0,1,1,1,1)^t$.
The morphism $\phi$ is harmonic and thus satisfies 
$A_2 \Phi = D_\nu \Phi + D_\mu \Phi A_1$. 
In this case, the degree of $\phi$ is 2 since each edge in $\Gamma_1$ has two preimages in $\Gamma_2$.
\end{example}

%\todo{I moved this part of the example here. We say that you can see $\phi$ is %harmonic because it satisfies the identity, but the stated proposition has the reverse %implication. Should we just note that the equality holds because $\phi$ is harmonic?}

The following lemma is an immediate consequence of the fact that the number of preimages of an edge is the same for each edge in $\Gamma_1$.  

\begin{lemma}
\label{lemma:surjective-harmonic-morphism}
A non-constant harmonic morphism is surjective on 
vertices and edges.
\end{lemma}

Furthermore, Baker and Norine \cite{BN09} showed 
that for each vertex $x$ in $\Gamma_1$, the sum of 
the horizontal multiplicities of the vertices in $\phi^{-1}(x)$ is equal 
to the degree of $\phi$. This fact is expressed in matrix 
form in the following lemma.

\begin{lemma}
\label{lemma:degree-harmonic-morphism}
%The degree of a non-constant harmonic morphism satisfies 
If $\phi$ is a non-constant harmonic morphism, then 
%the matrix identity
\[
\Phi^t D_\mu \Phi = \deg(\phi) I,
\]
where $I$ is the identity matrix of size $n_1 \times n_1$.
\end{lemma}
 
 We conclude this section with a technical linear algebra 
 lemma which will be used in the proof of 
 Theorem \ref{thm:harmonic-L-identity}.

 \begin{lemma}
 \label{lemma:diagonal-technical}
  Let $\Phi$ be an $n_2 \times n_1$ zero-one matrix 
 with the property that each row contains exactly one 1, 
% and all other entries in the row are 0, 
and let $\boldsymbol c = 
 (c_1 , \dots , c_{n_1})^t$ be a length $n_1$ vector of real numbers.
  The diagonal matrices $\Diag (\Phi \boldsymbol c)$ 
 and $\Diag(\boldsymbol c)$ are related by the identity
 \[
 \Diag(\Phi \boldsymbol c) \Phi = \Phi \Diag(\boldsymbol c).
 \]
 \end{lemma}
 
 \begin{proof} Let $\sigma(i)$ be the position of the 
 entry 1 in the $i$th row of $\Phi$.
 Then the $i$th diagonal element of $\Diag (\Phi \boldsymbol c)$
 is $c_{\sigma(i)}$.
 The $(i,j)$th entry of $\Diag (\Phi \boldsymbol c)\Phi$ is
 \[
 (\Diag (\Phi \boldsymbol c)\Phi)_{i,j} = 
 \begin{cases}
 c_{\sigma(i)} & \quad \text{if $j=\sigma(i)$,}\\
 0 &\quad \text{otherwise.}
 \end{cases}
 \]
 On the other hand, the $(i,j)$th entry of $\Phi \Diag(\boldsymbol c)$
 is 
 \begin{align*}
 (\Phi \Diag(\boldsymbol c))_{i,j}
& =\Phi_{i,j} c_j\\
&=
\begin{cases}
 c_{\sigma(i)} & \quad \text{if $j=\sigma(i)$,}\\
 0 &\quad \text{otherwise.}
 \end{cases}
 \end{align*}
 \end{proof}
 
\subsection{Arithmetical structures on graphs}\label{sec:arith structures}
An arithmetical structure on a connected graph $\Gamma=(V,E)$ with at least two vertices and with no loops
is a pair $(\R,\S)$ of maps $\R \colon V \rightarrow \zzz_+$
and $\S \colon V \rightarrow \zzz_+$ such that 
the greatest common divisor of the values $\R(v)$ for $v \in V$ 
is 1, and such that
\[
\S(v) =\frac 1 {\R(v)} \sum_{w \sim v} \R(w).
\]
 The arithmetical structure with $\R(v)=1$ and $\S(v) = \deg(v)$ 
for all $v$ is called the {\it natural arithmetical structure}. We will often identify a map $V\to \zzz_+$ with its vector of values 
(with respect to a fixed ordering of the vertices). We also note that it is more common in the literature to use $(\R, D)$ as notation for these maps; we use $(\R,\S)$ here to avoid a conflict of notation in the paper.

To each arithmetical structure $(\R,\S)$ on a graph $\Gamma$, we can associate a 
symmetric matrix $L=L(\Gamma;\S)$ called the 
\textit{arithmetical Laplacian matrix}, defined to be 
\[L = \Diag(\S) - A,\] 
where $A$ is the adjacency matrix of $\Gamma$. 
It follows from the relationship described above that $LR =\boldsymbol{0}$. 
Furthermore, Lorenzini proved in  \cite[Prop.~1.1]{Lo89} that $L(\Gamma;\S)$ is always a matrix of rank $n-1$. We note that the values of $\S$ are completely determined by $\R$, as seen in the formula above.
Similarly, the values of $\R$ are determined by $\S$, since $\R$ is the unique element of the kernel of $L$ with positive integer entries and with $\gcd(R)= 1$.

We say that a \textit{divisor} $\delta$ on a graph $\Gamma=(V,E)$ is a function 
$\delta \colon V \rightarrow \zzz$.  The {\it degree} of 
a divisor $\delta$ with respect to an arithmetical structure 
$(\R, \S)$ on $\Gamma$ is 
\[
\deg(\delta) = \sum_{v \in V} \delta(v) R(v).
\]
In other words, the degree of a divisor is $\delta$ is the dot product of $\delta$ and $R$. 

Divisors on $\Gamma$ form a group under addition, denoted by $\Div(\Gamma)$, and the divisors with degree~0 form a subgroup, denoted by $\Div^0(\Gamma)$.
If $f \colon V \rightarrow \zzz$ is an integer-valued function 
on vertices, we can define a divisor 
$\divisor(f)$ by
\[
\divisor(f)(v) =\S(v)f(v) - \sum_{w \sim v} f(w),
\]
i.e., $\divisor(f) = Lf$. 
Divisors of the form $\divisor(f)$ are called \textit{principal}.
A principal divisor has degree 0 
since $L$ is symmetric matrix, and therefore
$R^tLf=0$.
The set of principal divisors forms a 
subgroup $\Prin(\Gamma)$ of $\Div^0(\Gamma)$, and the quotient $\Div^0(\Gamma) / \Prin(\Gamma)$ 
is called the \textit{arithmetical critical group} of $\Gamma$ with respect 
to the arithmetical structure $(\R, \S)$, and is denoted 
$\mathcal K (\Gamma; \R, \S)$. This group is also isomorphic to the torsion part of the cokernel of $L$ as a map $\zzz^n\to\zzz^n$. 
Classifying arithmetical structures and their corresponding critical groups on various families of graphs has been the subject of study in several papers in recent years, including cycles and paths \cite{CV18,BCCGGKMMV18}, 
star graphs and complete graphs \cite{ADGL24}, 
bidents \cite{ABDGGL20}, 
%wedge sums of graphs \cite{CV18-2}, 
and several others.

Given a graph $\Gamma$ and an arithmetical structure $(\R,\S)$ on $\Gamma$, the arithmetical critical group can be computed  with computer algebra systems such as Mathematica or SageMath by computing the Smith normal form of $L(\Gamma;\S)$. Since the rank of $L(\Gamma;\S)$ is $n-1$, the diagonal of the Smith normal form is of the form $(e_1,e_2,\ldots, e_{n-1},0)$, where $e_1, e_2, \dots 
, e_{n-1}$ are nonzero integers, unique up to 
multiplication by $\pm 1$. 
The associated arithmetical critical group is \[\K(\Gamma;\R,\S) \cong \zzz/e_1\zzz \oplus\zzz/e_2\zzz \oplus \cdots \oplus \zzz/e_{n-1}\zzz.\] 
%The integers $|e_i|$ greater than 1 are called the 
%invariant factors.
A well-known fact about the Smith normal form is that $e_i\mid e_{i+1}$ and that $|e_i| = g_i/g_{i-1}$ where $g_i$ is the gcd of all $i\times i$ minors of $L(\Gamma;\S)$. Note that this implies that if there is an $(n-2)\times(n-2)$ minor of $L(\Gamma;\S)$ equal to $\pm1$, then $\K(\Gamma;\R,\S)$ is a cyclic group.

The critical group of $\Gamma$ with respect to the natural arithmetical structure, defined above to have $\R(v) = 1$ and $\S(v) = \deg(v)$ for all $v$, will be called 
the {\it natural critical group} of $\Gamma$. (This group is called the Jacobian in \cite{BN09}, and has also been called the Picard group \cite{BdlHN97,Bi99} or the sandpile group \cite{Dh90}.)
On a fixed graph, the critical group with respect to another arithmetical structure can 
be larger or smaller than the natural critical group. In \cite[Cor.~2.10]{Lo24}, Lorenzini shows that for any graph $\Gamma,$ there is some arithmetical structure on $\Gamma$ with trivial critical group.

\begin{example}
If $\Gamma$ is the wheel group $W_7$ with 
 central vertex $v_0$ and rim vertices 
 $v_1,  \dots , v_6$, 
then one can compute
the natural critical group of $\Gamma$ to be isomorphic to $\zzz / 8 \zzz \times 
\zzz / 40 \zzz$. If we let $\R = (1, 3, 1, 1, 3, 1, 1)^t$ and 
 $\S = (10, 1, 5, 5, 1, 5, 5)^t$, then the arithmetical critical group of 
$\Gamma$ with respect to $(\R, \S)$ is  $\K(\Gamma;\R,\S)\cong\zzz / 8 \zzz \times 
\zzz / 24 \zzz$, and if we let $\R' = (3, 1, 1, 1, 1, 1, 1)^t$ and 
 $\S' = (2, 5, 5, 5, 5, 5, 5)^t$, then the arithmetical critical group of 
$\Gamma$ with respect to $(\R', \S')$ is  $\K(\Gamma;\R',\S')\cong\zzz / 8 \zzz \times 
\zzz / 168 \zzz$.
\end{example}

\section{Homomorphisms of critical groups}\label{sec:main}

Given a non-constant harmonic morphism 
$\phi\colon \Gamma_2\to\Gamma_1$, we will define the pullback of an arithmetical structure on $\Gamma_1$ to an arithmetical structure on $\Gamma_2$. The harmonic morphism $\phi$ will determine homomorphisms, namely the pushforward $\phi_*$ and the pullback $\phi^*$, between the respective arithmetical critical groups associated to these structures. We prove the surjectivity and injectivity of $\phi_*$ and $\phi^*$, respectively. We will finish the section with some consequences of these results.

\subsection{The pullback of an arithmetical structure}

Let $\phi \colon \Gamma_2 \rightarrow \Gamma_1$ be a non-constant harmonic 
morphism of connected graphs, and 
let $(\R_1, \S_1)$ be an arithmetical structure on $\Gamma_1$.
We define the {\it pullback of $(\R_1, \S_1)$} to be the pair 
$\R_2 \colon V(\Gamma_2) \rightarrow \zzz_+$ 
and $\S_2 \colon V(\Gamma_2) \rightarrow \zzz_+$
given by 
\begin{equation}
\label{eqn:R2S2}
    \R_2(v) = R_1(\phi(v)) \quad \text{and} \quad
    \S_2(v) = \mu(v) S_1(\phi(v))+\nu(v),
\end{equation}
for $v \in V(\Gamma_2)$. 
In matrix form, $\R_2 = \Phi \R_1$
and $\S_2 =D_\mu \Phi \S_1 + \boldsymbol{\nu}$.

In the first theorem of this section, we show that 
$(\R_2,\S_2)$ is an arithmetical structure on $\Gamma_2$.
%We call it the {\it induced arithmetical structure} on $\Gamma_2$.
Equation (\ref{eqn:harmonic-L-identity}) generalizes
\cite[Proposition~3.3.25]{JM17} for the natural 
arithmetical structure.

\begin{theorem}\label{thm:harmonic-L-identity}
Let $\Gamma_1$ and $\Gamma_2$ be connected graphs, let $\phi \colon \Gamma_2 \rightarrow \Gamma_1$ be a 
non-constant harmonic morphism, and let $(\R_1, \S_1)$ be an arithmetical structure on $\Gamma_1$.
Then the pullback $(\R_2, \S_2)$ 
described in Equation (\ref{eqn:R2S2}) 
is an arithmetical structure on $\Gamma_2$.
Furthermore,
\begin{equation}
\label{eqn:harmonic-L-identity}
L_2 \Phi =D_\mu \Phi L_1.
\end{equation}
\end{theorem}

\begin{proof}
Notice that $\gcd(R_2)=\gcd(R_1)=1$.
To prove that $(\R_2,\S_2)$ is an 
arithmetical structure, it is enough to show
that $L_2 \R_2=\boldsymbol 0$, 
where $L_2 = \Diag(S_2) - A_2$.
This result will follow immediately from 
 Equation (\ref{eqn:harmonic-L-identity}), 
since $\R_2= \Phi \R_1$, and $L_1 \R_1 = \boldsymbol 0$, and thus $L_2R_2=L_2\Phi R_1=D_\mu \Phi L_1R_1=\boldsymbol{0}.$

We now prove that $L_2 \Phi =D_\mu \Phi L_1$.
%Recall that $L_2 = \Diag(S_2) - A_2$. 
Since we define $\S_2 =D_\mu \Phi \S_1 + \boldsymbol{\nu}$,
%From the definition of $S_2$,
\begin{equation*}
L_2 = D_\mu \Diag(\Phi \S_1) +  D_\nu - A_2.
\end{equation*}
By Lemma \ref{lemma:diagonal-technical},
 $\Diag(\Phi \S_1) \Phi = \Phi \Diag (\S_1)$.
By the adjacency matrix harmonic identity,
Proposition \ref{propn:harmonic-adjacency-criterion}, 
$(D_\nu  - A_2) \Phi = - D_\mu \Phi A_1$.
Therefore, 
\begin{align*}
L_2 \Phi 
&=D_\mu \Phi \Diag(\S_1) - D_\mu \Phi A_1\\
%&= D_\mu \Phi (\text{Diag}(\S_1) - A_1)\\
&= D_\mu \Phi L_1.
\end{align*}
%Therefore, $(\R_2, \S_2)$ is an arithmetical structure on $\Gamma_2$.
\end{proof}

We note that for any graphs $\Gamma_1$ and $\Gamma_2$ with 
non-constant harmonic morphism 
$\phi \colon \Gamma_2\to\Gamma_1$, the natural arithmetical structure on $\Gamma_1$ pulls back to the natural arithmetical structure on $\Gamma_2$.
Let us also consider the following example involving non-natural arithmetical structures.

\begin{example}
\label{ex:C3-W5-part2}
Let $\phi \colon \Gamma_2 \rightarrow \Gamma_1$ 
be the harmonic morphism 
described in Examples~\ref{ex:C3-W5-part1} and \ref{ex:C3-W5-part1.5}, 
mapping the wheel graph on 5 vertices to the 3-cycle graph.
%Let us demonstrate Theorem~\ref{thm:harmonic-L-identity}.
Let $\R_1 = (2,1,3)^t$ and $\S_1=(2,5,1)^t$. This defines an arithmetical
structure on $\Gamma_1$ since  
$L_1 \R_1 = \boldsymbol 0$ where \[L_1 = 
\left(
\begin{array}{rrr}
2&-1&-1\\ -1&5&-1 \\ -1&-1&1
\end{array} \right).\]
Notice that $R_2 = \Phi \R_1$ is given by
$\R_2 = (2,1,1,3,3)^t$
and $\S_2 = D_\mu\Phi S_1 + \boldsymbol{\nu}$ is given by $\S_2 = (4,6,6,2,2)^t$. 
In this case, we have that  $L_2 \Phi = D_\mu \Phi L_1$ and $L_2 \R_2 = \boldsymbol 0$, where 
\[
L_2 = 
\left(
\begin{array}{rrrrr}
4&-1&-1&-1&-1\\ -1&6&-1&0&-1\\
-1&-1&6&-1&0 \\ -1&0&-1&2&-1 \\ -1& -1&0&-1&2
\end{array} \right),
\]
and thus $(\R_2,\S_2)$ is also an arithmetical structure. 
%By calculating the Smith normal form of $L_1$ and $L_2$ 
%with Mathematica, we find that 
%the size of the critical group $\mathcal K(\Gamma_1)$ is 1,
%and the size of the critical group $\mathcal K(\Gamma_2)$ is 20.
\end{example}

\subsection{Surjectivity of the pushforward homomorphism}
Let $\phi \colon \Gamma_2 \rightarrow \Gamma_1$ 
be a non-constant harmonic morphism of connected 
graphs.  Recall that by Lemma \ref{lemma:surjective-harmonic-morphism},
such a map is surjective on vertices and edges.
We define a pushforward map 
$\phi_* \colon \Div(\Gamma_2) \rightarrow \Div(\Gamma_1)$ 
as follows.
If $\delta$ is a divisor on $\Gamma_2$, 
the {\it pushforward} divisor $\phi_* \delta$ on $\Gamma_1$ 
is given by 
\[
\phi_* \delta(x) = \sum_{v \in \phi^{-1}(x)} \delta(v)
\]
for each $x\in \Gamma_1$. In matrix form, we can write $\phi_* \delta = \Phi^t \delta$.
Since $\phi$ is surjective, the map 
$\phi_* \colon \Div(\Gamma_2) \rightarrow \Div(\Gamma_1)$, 
%viewed as a map from 
%the set of divisors on $\Gamma_2$ to the set of divisors on $\Gamma_1$, 
is also surjective.

Recall that the degree of a divisor $\delta$ on $\Gamma$ with respect to an arithmetical structure $(R,S)$ is given by $\deg(\delta) = R^t\delta$. 
Let $(R_1,S_1)$ be an arithmetical structure on $\Gamma_1$, 
and let $(R_2,S_2)$ be the pullback arithmetical structure on 
$\Gamma_2$, described in Equation (\ref{eqn:R2S2}). 

\begin{lemma}
\label{lemma:deg-pushforward}
The degree of the divisor $\phi_* \delta$ with respect to $(\R_1,\S_1)$ is equal to 
the degree of the divisor $\delta$ with respect to $(\R_2,\S_2)$.
\end{lemma}

\begin{proof}
Let $\Phi$ be the vertex map matrix of $\phi$ with respect 
to fixed orderings of the vertices of $\Gamma_2$ 
and $\Gamma_1$.
Recall from Equation (\ref{eqn:R2S2})
that  $R_2 = \Phi R_1$. The degree of a divisor $\delta$ on $\Gamma_2$ 
is $R_2^t \delta = R_1^t \Phi^t \delta$, and the degree of a divisor $\xi$ 
on $\Gamma_1$ is $R_1^t \xi$. 
Since the map $\phi_*$ is given in matrix form by $\phi_* \delta = \Phi^t \delta$, the degree 
of $\phi_*\delta$ is the same as the degree of $\delta$.
\end{proof}

The following theorem (and Theorem \ref{thm:injective}
below on injectivity)
are analogues of Baker and Norine's results for the natural critical groups \cite{BN09}.

\begin{theorem}
\label{thm:surjective-on-critical-groups}
Let $\phi \colon \Gamma_2 \rightarrow \Gamma_1$ 
be a non-constant harmonic morphism of connected graphs. 
Let $(R_1,S_1)$ be an arithmetical structure on $\Gamma_1$, 
and let $(R_2,S_2)$ be the pullback arithmetical structure on 
$\Gamma_2$, described in Equation (\ref{eqn:R2S2}).
Then $\phi$ induces a surjective homomorphism 
of arithmetical critical groups:
\[
\phi_* \colon \mathcal K(\Gamma_2; R_2, S_2) 
\rightarrow \mathcal K(\Gamma_1; R_1, S_1).
\]
\end{theorem}

\begin{proof}
By Lemma \ref{lemma:deg-pushforward},
if $\delta$ is a degree 0 divisor on $\Gamma_2$, 
then $\phi_* \delta$ is a degree 0 divisor on $\Gamma_1$.
Furthermore, since $\phi$ is surjective, every degree 0 
divisor on $\Gamma_1$ is the image of some degree 0 
divisor on $\Gamma_2$.

To show that $\phi_*$ determines a homomorphism of 
critical groups, we need to show that if $\delta$ is a principal 
divisor on $\Gamma_2$, i.e.,  a  
divisor of the form $\delta=L_2 f$ 
for some $n_2$-tuple $f$ of integers, then 
$\Phi^t \delta$ is principal on $\Gamma_1$, i.e., of the 
form $\Phi^t \delta =L_1 g$ for some $n_1$-tuple $g$ 
of integers. But by Theorem \ref{thm:harmonic-L-identity},
and the fact that $L_1$ and $L_2$ are symmetric matrices,
%\begin{equation*}
$\Phi^t L_2 f =L_1 \Phi^t D_\mu f .$ 
%\end{equation*}
Therefore, $\Phi^t \delta  = L_1 g$, 
where $g = \Phi^t D_\mu f$.

Therefore, $\phi$ determines a surjective homomorphism of critical groups 
\[
\phi_* \colon \K(\Gamma_2; R_2, S_2) 
\rightarrow \K(\Gamma_1; R_1, S_1).
\]
\end{proof}

\subsection{Injectivity of the pullback homomorphism}

Let $\phi \colon \Gamma_2 \rightarrow \Gamma_1$ 
be a non-constant harmonic morphism of connected graphs.
We define a pullback map $\phi^* \colon 
\Div(\Gamma_1)\rightarrow 
\Div(\Gamma_2)$ as follows.
If $\xi$ is a divisor on $\Gamma_1$, 
 the \textit{pullback} divisor $\phi^* \xi$ on $\Gamma_2$ is given by 
 \[
 \phi^*\xi(v) = \mu(v)\xi(\phi(v))
 \]
 for each $v\in\Gamma_2$.
In matrix form, we can write  $\phi^* \xi = D_\mu \Phi \xi$.

Let $(R_1,S_1)$ be an arithmetical structure on $\Gamma_1$, 
and let $(R_2,S_2)$ be the pullback arithmetical structure on 
$\Gamma_2$ with respect to $\phi$.

\begin{lemma}
\label{propn:degree-pullback}
The degree of $\phi^* \xi$ is equal to $\deg(\phi) \deg(\xi)$.
\end{lemma}

\begin{proof}
Recall that if $\xi \colon V(\Gamma_1) \rightarrow \mathbb Z$ is a 
divisor on $\Gamma_1$, the {\it degree} of 
$\xi$ is $\R_1^t \xi$.
%Recall that $\deg(\xi) = R_1^t \xi$.
The degree of the pullback of $\xi$ is 
\begin{align*}
\deg(\phi^* \xi) &= R_2^t (D_\mu \Phi \xi)\\
&= R_1^t \Phi^t D_\mu \Phi \xi.
\end{align*}
But by Lemma \ref{lemma:degree-harmonic-morphism},
$\Phi^t D_\mu \Phi = \deg(\phi) I$. Therefore, 
$\deg(\phi^* \xi) = \deg(\phi) \deg(\xi)$.
\end{proof}

The fact that $\phi$ pulls back divisor classes to divisor classes 
is a consequence of the next proposition.
%which states that if $\xi$ is a principal 
%divisor on $\Gamma_1$, then $\phi^* \xi$ is a principal
%divisor on $\Gamma_2$.

\begin{proposition}
If $\xi$ is a principal 
divisor on $\Gamma_1$, then $\phi^* \xi$ is a principal
divisor on $\Gamma_2$.
%If $\xi$ is an integer-valued $n_1$-vector  
%such that 
%$\xi= L_1 \eta$, for some integer-valued $\eta$, 
%then $\phi^* \xi = L_2 \zeta$, for some 
%integer-valued $n_2$-vector $\zeta$.
\end{proposition}

\begin{proof}
Suppose that $\xi$ is a principal 
divisor on $\Gamma_1$, i.e.,  a divisor 
of the form $\xi = L_1 g$, for some 
$n_1$-tuple $g$ of integers.
Then $\phi^*\xi =D_\mu \Phi L_1 g$, which 
is equal to $L_2 \Phi g$,
by Theorem \ref{thm:harmonic-L-identity}.
Let $f = \Phi g$. Then $\phi^* \xi 
=L_2 f$, so $\phi^* \xi$ is principal.
\end{proof}

%\todo{ Should we demonstrate Prop 2.2 also? Pull back and see that equivalent %divisors get sent to equivalent divisors?}

\begin{example}
\label{ex:C3-W5-part3}
Let $\phi \colon \Gamma_2 \rightarrow \Gamma_1$ be 
the harmonic morphism
of graphs with arithmetical structures described in Example 
\ref{ex:C3-W5-part2}. Let $\xi = (-4,5,1)^t$.
Then $\xi = L_1 g$, where $g = (1,2,4)^t$. 
We calculate that $\phi^* \xi = (-8,5,5,1,1)^t$.
Let $f = \Phi g$, i.e., $f = (1,2,2,4,4)^t$. Then 
$L_2f = (-8,5,5,1,1)^t = \phi^* \xi$.
\end{example}

%\todo{Baker and Norine also include that degree of pullback is degree of %original times degree of morphism. Is that still true with the arithmetical degree?}

The proof of the following theorem on injectivity of the pullback is more subtle than the proof of Theorem~\ref{thm:surjective-on-critical-groups} on surjectivity of the pushforward. 

\begin{theorem}\label{thm:injective}
Let $\phi \colon \Gamma_2 \rightarrow \Gamma_1$ 
be a non-constant harmonic morphism of connected graphs. 
Let $(R_1,S_1)$ be an arithmetical structure on $\Gamma_1$, 
and let $(R_2,S_2)$ be the pullback arithmetical structure on 
$\Gamma_2$.
%If $\xi$ is a degree 0 divisor on $\Gamma_1$ 
%such that the pullback of $\xi$ to $\Gamma_2$ is principal, %i.e., such that
%$D_\mu \Phi \xi = L_2 \eta$, for some 
%divisor $\eta$ on $\Gamma_2$, then $\xi$ is principal.
Then $\phi$ induces an injective homomorphism 
of arithmetical critical groups:
\[
\phi^* \colon \mathcal K(\Gamma_1; R_1, S_1) 
\rightarrow \mathcal K(\Gamma_2; R_2, S_2).
\]
\end{theorem}

\begin{proof}
Suppose that $\xi$ is a degree 0 divisor on $\Gamma_1$ 
such that $\phi^* \xi$ is a principal divisor 
on $\Gamma_2$, i.e., 
$D_\mu \Phi \xi = L_2 f$ 
for some $n_2$-tuple $f$ of integers.
%such that $\phi^* \xi=\text{div}(f)$ 
%for some function $f \colon V(\Gamma_2) \rightarrow \zzz$, or,
%in vector form, $D_\mu \Phi \xi = L_2 f$.
We will show that $\xi$ is a principal divisor on $\Gamma_1$.

First note that since $\xi$ is orthogonal to the kernel of 
$L_1$, there is a 
$\qqq$-valued $n_1$-tuple $\beta$
such that $\xi = L_1 \beta$.
Indeed, we can construct such a  $\beta$ as follows. 
Multiplying both sides of the equation
$D_\mu \Phi \xi = L_2 f$ on the left by $\Phi^t$, we obtain
$\Phi^t D_\mu \Phi \xi = \Phi^t L_2 f.$
By Lemma \ref{lemma:degree-harmonic-morphism},
 $\Phi^t D_\mu \Phi = d I$, 
where $d=\deg(\phi)$ and $I$ is the $n_1\times n_1$ identity matrix.
Furthermore, 
%by the harmonic morphism identity, 
by Theorem \ref{thm:harmonic-L-identity}
and the fact that $L_1$ and $L_2$ are symmetric, 
$\Phi^t L_2 = L_1 \Phi^t D_\mu$. Therefore,
\[d\xi = \Phi^tD_\mu\Phi\xi = \Phi^tL_2 f=  L_1 \Phi^t D_\mu f .\]
Let $\beta =  (1/ d) \Phi^t D_\mu f.$  
Then $\xi = L_1 \beta$.

Let $\alpha$ be 
$\qqq$-valued $n_2$-tuple 
$\alpha = \Phi \beta$. 
By Theorem \ref{thm:harmonic-L-identity},
$L_2 \alpha=D_\mu \Phi L_1 \beta$. 
Since $L_1 \beta =\xi$, we have $L_2 \alpha = D_\mu \Phi \xi$.
But we assumed that $D_\mu \Phi \xi= L_2 f$.
%where $f$ is integer-valued.
Since $L_2 \alpha = L_2 f$, we have that $f - \alpha = q \R_2$ 
for some rational number $q$ (the kernel of $L_2$ is spanned by $\R_2$).
Therefore, $f = \Phi( \beta + q \R_1)$.  
The function $f$ is integer-valued, 
and each row of $\Phi$ contains exactly one 1 and all other entries 0, and therefore, 
the $n_1$-tuple $g=\beta + q \R_1$ is also integer-valued.  In addition, 
$L_1g = L_1 \beta = \xi$, since 
$L_1 R_1 = \boldsymbol 0$, and 
thus, $\xi$ is principal.

Therefore, $\phi$ determines an injective homomorphism of critical groups 
\[
\phi^* \colon \K(\Gamma_1; R_1, S_1) 
\rightarrow \K(\Gamma_2; R_2, S_2).
\]
\end{proof}

\begin{example}\label{ex:wheel}
Let $\Gamma_1 = K_4$, the complete graph on 4 vertices
$x_0, x_1,x_2, x_3$.
Let $\Gamma_2 = W_7$, the wheel graph with central vertex $v_0$
and rim vertices $v_1, \dots , v_6$.
By identifying $x_0$ as the central vertex, 
we can view $K_4$ as the wheel graph $W_4$.
There is a harmonic morphism $\phi \colon \Gamma_2 
\rightarrow \Gamma_1$ that takes  $v_0$ to $x_0$ and
maps the 6 rim vertices of $\Gamma_2$ in a natural way 
to the 3 rim vertices of $\Gamma_1$. 
The horizontal multiplicity of $v_0$ is 2, and all other 
vertices of $\Gamma_2$ have horizontal multiplicity 1.
There are no vertical edges.

The arithmetical structure on $\Gamma_1$ defined by 
the 
pair $\R_1 = (1,1,1,3)^t$ and $\S_1 = (5,5,5,1)^t$
has critical group $\K(\Gamma_1;\R_1,\S_1)\cong\zzz /2 \zzz \times \zzz /6 \zzz$.
The pullback arithmetical structure 
%$(R_2,S_2)$ 
on $\Gamma_2$, defined by the pair
$R_2=(1,1,1,3,1,1,3)^t$ and $S_2 = (10,5,5,1,5,5,1)^t$,  has %corresponding
%critical group 
$\K(\Gamma_2;\R_2,\S_2)\cong\zzz /8 \zzz \times \zzz /24 \zzz$.

As another example, the pair $\R_1 = (3,1,1,1)^t$ and $\S_1 = (1,5,5,5)^t$
define an arithmetical structure on $\Gamma_1$ with  $\K(\Gamma_1;\R_1,\S_1)\cong\zzz /2 \zzz \times \zzz /6 \zzz$. 
The pullback arithmetical structure defined by the pair $R_2=(3,1,1,1,1,1,1)^t$ and $S_2 = (2,5,5,5,5,5,5)^t$ on $\Gamma_2$ has 
 $\K(\Gamma_2;\R_2,\S_2)\cong\zzz /8 \zzz \times \zzz /168 \zzz$.

If we form a new graph $\Gamma_2^\prime$ from 
$\Gamma_2$ by adding an edge between a pair of 
opposite rim vertices, the same map on vertices determines 
a harmonic morphism from $\Gamma_2^\prime$ to $\Gamma_1$ 
with one vertical edge.  
As noted above, the pair $\R_1 = (3,1,1,1)^t$ and $\S_1 = (1,5,5,5)^t$
define an arithmetical structure on $\Gamma_1$ with critical 
group $\K(\Gamma_1;\R_1,\S_1)\cong\zzz /2 \zzz \times \zzz /6 \zzz$.
The critical group of $\Gamma_2^\prime$ 
with respect to the pullback arithmetical structure $(R'_2,S'_2)$ in this case is 
$\K(\Gamma_2^\prime;\R'_2,\S'_2)\cong\zzz /4 \zzz \times \zzz /480 \zzz$, a very different answer than the one before.
\end{example}

\subsection{Consequences}

There are a few corollaries to the theorems in this section that we can state. The first corollary follows immediately from Theorem~\ref{thm:injective}.
\begin{corollary}
Let $(R_1,S_1)$ be an arithmetical structure on $\Gamma_1$ 
and let $(R_2,S_2)$ be the pullback arithmetical structure on 
$\Gamma_2$ by any non-constant harmonic morphism 
$\phi \colon \Gamma_2\to\Gamma_1$. Then, 
$|\mathcal{K}(\Gamma_1;R_1,S_1)|$ divides  $|\mathcal{K}(\Gamma_2;R_2,S_2)|$.

\end{corollary}

Recall that the critical group of an arithmetical structure is 
determined by the Smith normal form of the Laplacian matrix.
The absolute values of the 
diagonal elements of the Smith normal 
form that are neither zero nor a unit 
determine  the invariant factors of the critical group.
%We define the {\it complexity} of the arithmetical %critical group of a graph 
%to be the number of diagonal elements of the Smith normal 
%form that are neither zero nor a unit. 
%This is exactly the number of invariant factors the arithmetical critical group has.
For example, the natural critical group of the complete graph $K_n$ has $n-2$ invariant factors since it 
is isomorphic to $(\zzz / n \zzz)^{n-2}$. 
The cycle graph $C_n$ has cyclic natural critical group isomorphic to $\zzz / n \zzz$ with only 1 invariant factor. 
Each of the arithmetical critical groups in Example~\ref{ex:wheel} 
has 2 invariant factors.

\begin{corollary}\label{cor:complexity}
If a graph $\Gamma_1$  has an arithmetical 
critical group
%structure 
with  $\gamma_1$ invariant factors, and 
if every arithmetical 
critical group of
%structure on
a graph $\Gamma_2$ has
fewer than $\gamma_1$ invariant factors, then 
there is no non-constant harmonic morphism from $\Gamma_2$ to $\Gamma_1$.
\end{corollary}

Note that a similar theorem holds for the natural critical group; that is, if the 
number of invariant factors of the natural critical group of $\Gamma_1$ is greater than the 
number of invariant factors of the natural critical group of $\Gamma_2$, then there cannot
exist a non-constant harmonic morphism from $\Gamma_2$ to $\Gamma_1$.
This has no application when $\Gamma_1$ is a tree since the 
natural critical group of any tree is trivial.
%and thus has complexity 0.  
However, there are other arithmetical 
structures on trees that have non-trivial arithmetical critical groups.

%\begin{example}
%Let $\Gamma_1 = C_4$, let $\Gamma_2 = C_8$, and 
%let $\phi$ be a covering morphism, and let 
%$\Phi$ be the vertex map matrix (with respect to 
%the natural orderings of the vertices of $C_4$ and $C_8$).
%Let $\R_1 = (1,2,3,1)$, and let $\R_2 = \Phi R_1$.
%The size of the critical group of $\Gamma_1$ is 2 and 
%the size of the critical group of $\Gamma_2$ is 4.
%\end{example}

\begin{example}
Let $\Star_n$ be the star graph on $n$ vertices 
with central vertex $v_0$ of degree $n-1$ and spoke vertices 
$v_1, \dots , v_{n-1}$ of degree 1. 
All arithmetical structures on $\Star_4$ are cyclic since there is a $2\times 2$ 
minor of its Laplacian (regardless of the entries of its diagonal) that is equal to 1. For example, the graph $\Star_4$ with 
arithmetical structure $R=(3,1,1,1)^t$, 
$S=(1, 3,3,3)^t$ has critical group $\K(\Star_4;\R,\S)\cong\zzz / 3 \zzz$.

%\begin{itemize}
%\item The graph $\Star_4$ with arithmetical structure $R=(4,1,1,2)^t$, 
%$S=(1, 4, 4, 2)^t$ has critical group $\K(\Star_4;\R,\S)\cong\zzz / 2 \zzz$.
%\item The graph $\Star_4$ with arithmetical structure $R=(3,1,1,1)^t$, 
%$S=(1, 3,3,3)^t$ has critical group $\K(\Star_4;\R,\S)\cong\zzz / 3 \zzz$.
%\item The graph $\Star_5$ with arithmetical structure $R=(4,1,1,1,1)^t$, 
%$S=(1, 4,4,4,4)^t$ has critical group $\K(\Star_5;\R,\S)\cong\zzz / 4 \zzz %\times \zzz / 4 \zzz$.
%\item The graph $\Star_5$ with arithmetical structure $R=(6,1,1,2,2)^t$, 
%$S=(1,6,6,3,3)^t$ has critical group $\K(\Star_5;\R,\S)\cong\zzz / 3 \zzz %\times \zzz / 3 \zzz$.
%\end{itemize}

However, there are non-cyclic critical groups on $\Star_5$.
For example, 
the graph $\Star_5$ with arithmetical structure $R=(6,1,1,2,2)^t$, 
$S=(1,6,6,3,3)^t$ has critical group $\K(\Star_5;\R,\S)\cong\zzz / 3 \zzz 
\times \zzz / 3 \zzz$.
%, as demonstrated above. 
In fact, if follows from \cite[Prop.~16]{ADGL24} that there is an arithmetical critical group on $\Star_n$ 
with $m$ invariant factors 
%of complexity $m$ 
for any $0\leq m\leq n-3.$ 
\end{example}

In the next example, we will see how the Corollary~\ref{cor:complexity} can be used to show that no harmonic morphism exists from certain graphs to $\Star_5$ (or any other graph that admits non-cyclic critical groups).

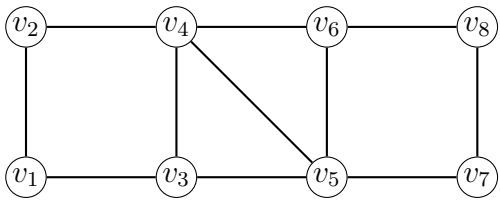
\begin{figure}[ht]
\begin{minipage}{\textwidth}
\begin{center}
%\vspace{8.0 cm}
%\includegraphics[scale=0.75]{graph_with_cyclic_K.png}
\begin{tikzpicture}
\node[circle, draw,inner sep=1pt] (v1) at (0,0) {$v_1$};
\node[circle, draw,inner sep=1pt] (v2) at (0,2) {$v_2$};
\node[circle, draw,inner sep=1pt] (v3) at (2,0) {$v_3$};
\node[circle, draw,inner sep=1pt] (v4) at (2,2) {$v_4$};
\node[circle, draw,inner sep=1pt] (v5) at (4,0) {$v_5$};
\node[circle, draw,inner sep=1pt] (v6) at (4,2) {$v_6$};
\node[circle, draw,inner sep=1pt] (v7) at (6,0) {$v_7$};
\node[circle, draw,inner sep=1pt] (v8) at (6,2) {$v_8$};

 \draw [thick] (v1)--(v2) --(v4)--(v6)--(v8)--(v7)--(v5)--(v6); 
  \draw [thick] (v1)--(v3)--(v4)--(v5)--(v3);
 
\end{tikzpicture}
\end{center}
\end{minipage}
\caption{A graph with only cyclic critical groups.}
\label{fig:only_cyclic}
\end{figure}

\begin{example}
The graph in Figure \ref{fig:only_cyclic} has 
only cyclic arithmetical critical groups. 
This can be seen by considering
the Laplacian matrix associated to the arithmetical structure $(\R,\S)$ with 
$\S = (s_1, \dots , s_8)^t$,
\[
L=\left(
\begin{array}{rrrrrrrr}
 s_1 & -1 & -1 & 0 & 0 & 0 & 0 & 0 \\
 -1 & s_2 & 0 & -1 & 0 & 0 & 0 & 0 \\
 -1 & 0 & s_3 & -1 & -1 & 0 & 0 & 0 \\
 0 & -1 & -1 & s_4 & -1 & -1 & 0 & 0 \\
 0 & 0 & -1 & -1 & s_5 & -1 & -1 & 0 \\
 0 & 0 & 0 & -1 & -1 & s_6 & 0 & -1 \\
 0 & 0 & 0 & 0 & -1 & 0 & s_7 & -1 \\
 0 & 0 & 0 & 0 & 0 & -1 & -1 & s_8 \\
\end{array}
\right).
\]
Notice that the $6\times6$ submatrix one obtains by considering the first 6 rows and last 6 columns is a lower triangular matrix with determinant 1. Thus, the Smith normal form of $L$ will have at most one nonzero entry not equal to a unit. 
It follows from Corollary~\ref{cor:complexity} that this graph cannot 
map harmonically onto the star graph $\Star_5$ 
on 5 vertices. Furthermore, for a similar reason, if $\Gamma$ is any connected graph on $n$ vertices and there exists some ordering of the vertices such that for all $1\leq i\leq n-2$, $v_i\sim v_{i+2}$, and such that $v_i\not\sim v_j$ if $j>i+2$, then $\Gamma$ must only admit cyclic critical groups, and thus there can be no harmonic morphism from $\Gamma$ to $\Star_5$.

As noted, such a claim cannot be made using the natural critical group, since the natural critical group on any tree is trivial.
\end{example}

\section{Riemann-Hurwitz Formula}\label{sec:RH}

%\subsection{Riemann-Hurwitz formula}

In this section, we show that there is a Riemann-Hurwitz formula for 
graphs with arithmetical structures, using the same  ramification divisor used by Baker and Norine \cite{BN09}.
We use the arithmetical analogues of  the degree of a divisor on a graph, the canonical 
divisor for a graph, and the genus of a graph in the statement and proof of the theorem.

Let $\Gamma = (V,E)$ be a connected graph of order $n$ with an 
arithmetical structure $(R,S)$. 
We define the \textit{arithmetical canonical divisor $K_\Gamma$} on $\Gamma$ to 
be the divisor given by
\[
K_\Gamma (v) = S(v) - 2
\]
at each vertex $v\in V(\Gamma).$
If we fix an ordering on 
the vertices of  $\Gamma$, we can identify $K_\Gamma$ 
with the vector $K_\Gamma = S -2 \boldsymbol 1_n$,
where $\boldsymbol 1_n$ is the $n$-vector of 1's.
Note that for the natural arithmetical structure
we have $S(v) = \deg(v)$, and thus in this case our formula 
for $K_\Gamma$ reduces to the definition used by Baker and Norine.

Let $\phi \colon \Gamma_2 \rightarrow \Gamma_1$ be 
a non-constant harmonic morphism of connected graphs, 
with fixed orderings 
on the vertices 
of $\Gamma_1$ and $\Gamma_2$, and with $n_1=|V(\Gamma_1)|$ and $n_2=|V(\Gamma_2)|$.
The {\it ramification divisor of $\phi$} is defined to be the divisor on $\Gamma_2$ 
given by
\[
\Ram_\phi(v) = 2\mu(v)-2+\nu(v),
\]
i.e., 
$\Ram_\phi =2 \boldsymbol \mu - 2 \boldsymbol 1_{n_2} + \boldsymbol \nu$,
where $\boldsymbol \mu$ is the vector of horizontal multiplicities 
and $\boldsymbol \nu$ is the vector of vertical multiplicities of 
$\phi$. 

Let $(\R_1,\S_1)$ be an arithmetical structure on $\Gamma_1$, 
and let $(\R_2,\S_2)$ be the pullback arithmetical structure on 
$\Gamma_2$.
In the next proposition, we give a relationship 
between the arithmetical canonical divisors on 
$\Gamma_1$ and $\Gamma_2$.

\begin{proposition}
\label{propn:K-Ram}
The arithmetical canonical divisors $K_{\Gamma_1}$ and 
$K_{\Gamma_2}$ satisfy
%of $\Gamma_1$ and 
%$\Gamma_2$ are related by the identity
\[
K_{\Gamma_2} = \phi^* K_{\Gamma_1} + \Ram_\phi.
\]
\end{proposition}

\begin{proof}
Recall that the pullback of the arithmetical canonical divisor 
$K_{\Gamma_1}$ is given in vector form 
by 
$
\phi^* K_{\Gamma_1} = D_\mu \Phi K_{\Gamma_1},
$
where $D_\mu$ is the diagonal matrix of horizontal 
multiplicities of $\phi$, and $\Phi$ is the vertex map 
matrix of $\phi$.
We wish to show that 
$
K_{\Gamma_2} = D_\mu \Phi K_{\Gamma_1} + \Ram_\phi,
$
where $K_{\Gamma_i} = \S_i - 2 \boldsymbol 1_{n_i}$,
for $i\in\{1,2\}$.
%and $\Ram_\phi = 
%2 \boldsymbol \mu - 2 \boldsymbol 1_{n_2} + \boldsymbol \nu$.
Recall from Equation (\ref{eqn:R2S2})
that $\S_2 =D_\mu \Phi \S_1 + \boldsymbol{\nu}$.
Therefore,
\begin{align*}
D_\mu \Phi K_{\Gamma_1} + \Ram_\phi 
&= D_\mu \Phi (\S_1 - 2 \boldsymbol 1_{n_1})
+2 \boldsymbol \mu - 2 \boldsymbol 1_{n_2} + \boldsymbol \nu \\
&= \S_2 - 2 \boldsymbol 1_{n_2}
 - 2 D_\mu \Phi \boldsymbol 1_{n_1} +2 \boldsymbol \mu \\
 &= K_{\Gamma_2} 
  - 2 D_\mu \Phi \boldsymbol 1_{n_1} +2 \boldsymbol \mu .
 \end{align*}
 But $\Phi \boldsymbol 1_{n_1} = \boldsymbol 1_{n_2}$, since 
 $\Phi$ has exactly one 1 in every row, and $D_\mu 
  \boldsymbol 1_{n_2} = \boldsymbol \mu$, since $D_\mu = \Diag 
  (\boldsymbol \mu)$. 
  Therefore, $D_\mu \Phi K_{\Gamma_1} + \Ram_\phi 
  =K_{\Gamma_2} $.
\end{proof}

We define the {\it arithmetical genus $g$}
 of a connected graph $\Gamma$ 
with arithmetical structure $(\R,\S)$ by
\[
2g-2 = \deg(K_\Gamma),
\]
i.e., $2g - 2 = \sum_v \R(v) (S(v) - 2)$. 
%Lorenzini calls $g$ the linear rank of $\Gamma$.
Note that for the natural arithmetical structure, 
this is exactly what Baker and Norine called the genus of the graph, equal to  $|E(G)|-|V (G)|+1$ (also called the first Betti number).
This definition of arithmetical genus differs slightly from Lorenzini's definition of linear rank, 
which replaces $\S(v)$ by $\deg(v)$. However, the 
two definitions are equivalent by Lemma \ref{lemma:S-deg}
below.
Lorenzini shows that the linear rank (and thus the arithmetical genus $g$) is at least as large 
as the natural genus of the graph, and thus since $\Gamma$ is connected, we have $g \geq 0$ \cite[Theorem 4.7]{Lo89}.

\begin{lemma}\label{lemma:S-deg}
If $(\R,\S)$ is an arithmetical structure on $\Gamma$, then 
\[
\sum_v \R(v)\S(v) = \sum_v \R(v) \deg(v).
\]
\end{lemma} 

\begin{proof}
We will use the facts that $L = \Diag(S) - A$
and $L\R = \boldsymbol 0$.
Let $\boldsymbol 1$ be the all 1's 
vector of size equal to the order of $\Gamma$, and let 
$\boldsymbol d$ be the vector of degrees of vertices of $\Gamma$.
Since $\Diag(S) \boldsymbol 1=\S$ and $A \boldsymbol 1 = \boldsymbol d$,
it follows that $L \boldsymbol 1 =\S - \boldsymbol d$.
Recall that $L$ is symmetric, so the identity 
$LR = \boldsymbol 0$ is equivalent to 
$R^t L = \boldsymbol 0$.
Therefore, $\R^t L\boldsymbol 1 = 0$, and
$\R^t S = \R^t \boldsymbol d.$
This completes the proof of the lemma.
\end{proof}

We now show that there is a Riemann-Hurwitz formula 
for arithmetical graphs, corresponding to Baker and Norine's 
Riemann-Hurwitz formula for graphs \cite{BN09}. 
The proof is similar to that of Baker and Norine
for graphs with the natural arithmetical structure.

\begin{theorem}\label{thm:RH}
Let $\phi \colon \Gamma_2 \rightarrow \Gamma_1$ 
be a non-constant harmonic morphism of connected graphs.
Let $(\R_1,\S_1)$ be an arithmetical structure on $\Gamma_1$,
and let $(\R_2,\S_2)$ be the pullback arithmetical structure on $\Gamma_2$
given in Equation (\ref{eqn:R2S2}).
%on $\Gamma_2$, where $\R_2 = \Phi \R_1$.
% and $\Phi$ is 
%the vertex map matrix of $\phi$ with respect to fixed 
%orderings on the vertices of $\Gamma_1$ and $\Gamma_2$.
Let $g_i$ be the arithmetical genus 
of $\Gamma_i$ for $i\in\{1,2\}$. Then
\[
2g_2 - 2 = \deg(\phi) (2g_1 - 2) 
+ \sum_v \R_2(v) \left( 2 \mu (v) - 2 + \nu (v) \right).
\]
\end{theorem}

\begin{proof}
The identity to be proved can be restated as 
\[
\deg(K_{\Gamma_2}) = \deg(\phi) \deg (K_{\Gamma_1}) 
+ \deg(\Ram_\phi).
\]
By Proposition \ref{propn:K-Ram},
$K_{\Gamma_2} = \phi^* K_{\Gamma_1} + \Ram_\phi$.
Furthermore, $\deg( \phi^* K_{\Gamma_1}) 
=  \deg(\phi) \deg (K_{\Gamma_1})$ from Lemma~\ref{propn:degree-pullback}, so the result follows.
%Recall that  $\phi^* K_{\Gamma_1} = D_\mu \Phi K_{\Gamma_1}$. 
%Therefore,
%$\deg( \phi^* K_{\Gamma_1}) =R_2^t D_\mu \Phi K_{\Gamma_1} $.
%Since $R_2^t = R_1^t \Phi^t$, therefore
%\[
%\deg( \phi^* K_{\Gamma_1})=
%R_1^t ( \Phi^t D_\mu \Phi ) K_{\Gamma_1}.
%\]
%By Lemma \ref{lemma:degree-harmonic-morphism},
%$\Phi^t D_\mu \Phi = \deg(\phi) I$.
%which gives $\deg( \phi^* K_{\Gamma_1})  = \deg(\phi) 
%R_1^t K_{\Gamma_1}$.
%Recall that $R_1^t K_{\Gamma_1}=\deg ( K_{\Gamma_1})$. 
%Therefore, $\deg( \phi^* K_{\Gamma_1}) 
%=  \deg(\phi) \deg (K_{\Gamma_1})$.
%This completes the proof of the Riemann-Hurwitz formula for 
%arithmetical graphs.
\end{proof}

As a corollary of the Riemann-Hurwitz theorem, 
we prove the following 
analogue of Baker and Norine's result for natural 
arithmetical structures.

\begin{corollary} \label{cor:g1g2}
Let $\phi \colon \Gamma_2 \rightarrow \Gamma_1$ 
be a non-constant harmonic morphism of connected graphs.
Let $(R_1,S_1)$ be an arithmetical structure on $\Gamma_1$, 
and let $(R_2, S_2)$ be the pullback arithmetical structure on 
$\Gamma_2$. Then $g_2 \geq g_1$.
\end{corollary}

\begin{proof}
%Recall that the Riemann-Hurwitz formula tells us that
%\[
%2g_2-2 = \deg(\phi) (2g_1 - 2) 
%+ \sum_{v \in V(\Gamma_2)} R_2(v)\left( 2 \mu(v) - 2 + \nu(v) \right).
%\]
Since $R_2(v) = R_1(\phi(v))$, we can rewrite the summation term
in the Riemann-Hurwitz formula as 
\[
 %\sum_{v \in V(\Gamma_2)} 
 \sum_v
 R_2(v)\left( 2 \mu(v) - 2 + \nu(v) \right)
=
\sum_{x \in V(\Gamma_1)} R_1(x) \sum_{v \in \phi^{-1}(x)}
\left( 2 \mu(v) - 2 + \nu(v) \right).
\]
If $v$ is a vertex in $\phi^{-1}(x)$ which is not 
incident to any vertical edge, then $\nu(v)=0$, and, 
since $\Gamma_2$ is connected, $\mu(v) \geq 1$. 
Consequently, for such $v$, $2 \mu(v) - 2 + \nu(v)  \geq 0$.

Suppose $\phi^{-1}(x)$ contains at least one 
vertical edge.
Let $G_x=(V_x,E_x)$ be the subgraph of $\Gamma_2$ 
consisting of all vertical edges in $\phi^{-1}(x)$ together 
with their incident vertices. 
%Notice that $V_x$ is the 
%set of vertices in $v \in \phi^{-1}(x)$ such that 
%$\nu(v) \geq 1$.
The first Betti number of $G_x$ is 
$\beta_x =c_x - |V_x|+ |E_x| $, where $c_x$ is the 
number of connected components of $G_x$.
Since $\beta_x \geq 0$, we have
$|E_x| \geq -c_x+ |V_x| $.
Notice that 
$\sum_{v \in V_x} \nu(v)=2|E_x|$.
Consequently,
\begin{align*}
\sum_{v \in V_x}
\left( 2 \mu(v) - 2 + \nu(v) \right) 
&\geq 
-2c_x+2|V_x|
+ \sum_{v \in V_x} \left( 2 \mu(v)
- 2\right)\\
&= -2c_x +
\sum_{v \in V_x} 2 \mu(v).
\end{align*}
Since $\Gamma_2$ is connected, 
at least one vertex $v$ in each 
connected component of
$G_x$ must have 
horizontal multiplicity $\mu(v) \geq 1$.
Therefore,
\[
\sum_{v \in V_x}
\left( 2 \mu(v) - 2 + \nu(v) \right) 
\geq 0.
\]

Thus, it follows from Theorem \ref{thm:RH}
that 
$
2g_2 - 2 \geq \deg(\phi) (2g_1 - 2).
$
If $g_1 = 0$, then $g_2 \geq g_1$ holds trivially 
by Lorenzini's result that arithmetical genus 
(linear rank) is nonnegative. 
If $g_1 \geq 1$, then 
$
2g_2 - 2 \geq \deg(\phi) (2g_1 - 2) \geq 2g_1 - 2
$
so again $g_2 \geq g_1$.
\end{proof}

%\todo{Does this imply that $g_2>g_1$?}

\begin{example}
Let us demonstrate the concepts in this section with the arithmetical structure from Example~\ref{ex:C3-W5-part2}. In that example, we have $R_1=(2,1,3)^t$ and $S_1=(2,5,1)^t$ on the cycle graph $\Gamma_1 = C_3$, and $R_2=(2,1,1,3,3)^t$ and $S_2=(4,6,6,2,2)^t$ on the wheel graph $\Gamma_2 = W_5$. In this case, the arithmetical cannonical divisor $K_{\Gamma_1} = (0,3,-1)^t$ has degree $R_1^tK_{\Gamma_1} =0$, and $K_{\Gamma_2}=(2,4,4,0,0)^t$ has degree $R_2^tK_{\Gamma_2} =12$. This implies that $g_1 = 1$ and $g_2=7$. The degree of the harmonic morphism 
$\phi \colon \Gamma_2\to\Gamma_1$ is 2 and its 
ramification divisor is equal to 
\[\Ram_\phi=2(2,1,1,1,1)^t-2(1,1,1,1,1)^t+(0,1,1,1,1)^t=(2,1,1,1,1)^t.\]
We can check that 
 \[
 12 = 2(0) + (2,1,1,3,3)(2,1,1,1,1)^t,
 \] thus demonstrating Theorem~\ref{thm:RH}.
\end{example}

\subsubsection*{Disclaimer} The views expressed in this article do not necessarily represent 
the views or opinions of the U.S. Naval Academy, Department of the Navy, or Department of Defense or any of its components.

\end{document}